\documentclass[reqno,11pt]{amsart}
\usepackage[letterpaper, margin=1in]{geometry}
\RequirePackage{amsmath,amssymb,amsthm,graphicx,mathrsfs,url,slashed,subcaption}
\RequirePackage[usenames,dvipsnames]{xcolor}
\RequirePackage[colorlinks=true,linkcolor=Red,citecolor=Green]{hyperref}
\RequirePackage{amsxtra}
\usepackage{cancel}
\usepackage{tikz-cd}

\title{Regularity of sets of least perimeter in Riemannian manifolds}
\author{Aidan Backus}
\address{Department of Mathematics, Brown University}
\email{aidan\_backus@brown.edu}
\date{\today}

\newcommand{\NN}{\mathbf{N}}

\newcommand{\RR}{\mathbf{R}}

\newcommand{\Sph}{\mathbf S}

\newcommand{\Ball}{\mathbf{B}}

\DeclareMathOperator{\avg}{avg}

\DeclareMathOperator{\diam}{diam}

\DeclareMathOperator{\Exc}{Exc}

\DeclareMathOperator{\supp}{supp}

\newcommand*\dif{\mathop{}\!\mathrm{d}}

\DeclareMathOperator{\dist}{dist}

\newcommand{\Lagrange}{\mathscr L}

\newcommand{\normal}{\mathbf n}

\newcommand{\dfn}[1]{\emph{#1}\index{#1}}

\newcommand{\loc}{\mathrm{loc}}

\def\Japan#1{\left \langle #1 \right \rangle}

\newtheorem{theorem}{Theorem}[section]

\newtheorem{lemma}[theorem]{Lemma}

\newtheorem{proposition}[theorem]{Proposition}
\newtheorem{corollary}[theorem]{Corollary}

\theoremstyle{definition}
\newtheorem{definition}[theorem]{Definition}



\numberwithin{equation}{section}

\def\Xint#1{\mathchoice
{\XXint\displaystyle\textstyle{#1}}%
{\XXint\textstyle\scriptstyle{#1}}%
{\XXint\scriptstyle\scriptscriptstyle{#1}}%
{\XXint\scriptscriptstyle\scriptscriptstyle{#1}}%
\!\int}
\def\XXint#1#2#3{{\setbox0=\hbox{$#1{#2#3}{\int}$ }
\vcenter{\hbox{$#2#3$ }}\kern-.6\wd0}}

\def\dashint{\Xint-}

\usepackage[backend=bibtex,style=alphabetic,giveninits=true]{biblatex}

\renewbibmacro{in:}{}
\DeclareFieldFormat{pages}{#1}

\begin{document}
\begin{abstract}
    We give a short proof, in the tradition of the classical work of de Giorgi and Miranda on flat space, that the reduced boundary of a set of least perimeter in a Riemannian manifold of dimension $\leq 7$ is a smooth minimal hypersurface.
\end{abstract}

\maketitle

\section{Introduction}
An open subset $U$ of a Riemannian manifold-with-boundary $M$ is said to have \dfn{least perimeter} if its indicator function $1_U$ has least gradient -- that is, for every $\varphi \in BV(M)$ with $\varphi|_{\partial M} = 0$, 
\begin{equation}\label{least perimeter dfn}
\|\dif 1_U\|_{TV} \leq \|\dif 1_U + \dif \varphi\|_{TV},
\end{equation}
where $\|\cdot\|_{TV}$ denotes the total variation norm of a Radon measure.
In this paper we shall give a new proof of the following theorem:

\begin{theorem}\label{main thm}
Let $M$ be a Riemannian manifold-with-boundary of dimension $\leq 7$, and let $U \subset M$ have least perimeter.
Then $U$ is bounded by smooth, stable minimal hypersurfaces.
\end{theorem}

Here, a minimal hypersurface $N$ is \dfn{stable} if for every normal variation $(N_t)$ which has compact support on the interior of $M$, $\partial_t^2|_{t = 0} |N_t| \geq 0$.

Theorem \ref{main thm} was proven by de Giorgi and Miranda in the 1960s that if $M$ is an open subset of $\RR^d$, $d \leq 7$, and $U$ has least perimeter, then $U$ is bounded by minimal hypersurfaces \cite{deGiorgi61, Miranda66}.
The following argument suffices to prove Theorem \ref{main thm} even if $M$ is curved: taking a Nash embedding and rescaling, we may assume that $\partial U$ is an integral varifold whose mean curvature is small in $L^\infty$, so that $\partial U$ satisfies Allard's $\varepsilon$-regularity theorem \cite{DeLellis18}.
It follows that either $\partial U$ is smooth, or $\partial U$ has a singular tangent cone; the latter case implies $d \geq 8$.

However, one expects a direct generalization of Miranda's argument, and that is our goal in this present paper.
Miranda's key insight was that a $\varepsilon$-regularity theorem based on the oscillation of the normal vector, known as \dfn{de Giorgi's lemma}, can be deduced for sets of least perimeter using a monotonicity formula for functions of least gradient \cite{Miranda66}.
This argument is specific to codimension $1$ and completely avoids the geometric measure theory of varifolds (or currents) but uses the older and simpler theory of $BV$ functions.

I would like to highlight my upcoming work \cite{BackusCML,BackusInfinityMaxwell1} which studies the connection between functions of least gradient, laminations of minimal surfaces, and calibrations which solve certain generalizations of the $\infty$-Laplace equation.
This theory is specific to codimension $1$ and uses $BV$ functions and Theorem \ref{main thm} in an essential way; therefore, it is natural to want a proof of Theorem \ref{main thm} that also is based on $BV$.


\subsection{Idea of the proof}
For an exposition of Miranda's proof of Theorem \ref{main thm} in the flat case, we refer the reader to \cite[Chapters 5-9]{Giusti77} as well as Miranda's original work \cite{Miranda66}.
Consider the \dfn{excess} $\Exc_\Omega U$ -- that is, the oscillation of the conormal $1$-form $\normal_U$ to the reduced boundary $\partial^* U$ of $U$, measured in an open set $\Omega$.
By definition,
\begin{equation}\label{intro excess}
\Exc_\Omega U := |\partial^* U \cap \Omega| - \left|\int_\Omega \normal_U \dif \mu\right|
\end{equation}
where $\mu$ is the surface measure on $\partial^* U$ and the integral is a vector-valued integral.
By a mollification argument using a monotonicity formula for functions of least gradient, it suffices to estimate this quantity when $\partial^* U$ is the graph of a $C^1$ solution $u$ of the minimal surface equation\footnote{Strictly speaking, $u$ need only be an approximate solution of the minimal surface equation, since the relevant mollification operator does not commute with the minimal surface operator.}
\begin{equation}\label{euclidean MSE}
\nabla \cdot \frac{\nabla u}{\sqrt{1 + |\nabla u|^2}} = 0,
\end{equation}
such that $\nabla u$ is small in $C^0$. 
Linearizing (\ref{euclidean MSE}) around the trivial solution, we obtain the Laplace equation; from this, we deduce that we may approximate $u$ by a sum of orthogonal harmonic polynomials, which then implies the estimate 
\begin{equation}\label{intro DGL}
\Exc_{B(P, r/2)} U \leq 2^{-d} \Exc_{B(P, r)} U.
\end{equation}
The estimate (\ref{intro DGL}) is called the \dfn{de Giorgi lemma}, and by the aforementioned mollification argument, holds even if $\partial^* U$ is not assumed $C^1$.
It is easy to see that the de Giorgi lemma implies that $\partial^* U$ is actually $C^1$; the theorem then follows from (\ref{least perimeter dfn}).

While the above scheme roughly is also the plan of this paper, it does not quite work as stated above.
There are a handful of technicalities caused by the presence of curvature; here we just address the most important one.

The definition (\ref{intro excess}) assumes that $\normal_U$ is a map into a fixed vector space, and makes no sense if $\normal_U$ must be understood as a section of a vector \emph{bundle}; in other words, it makes no sense if $M$ has curvature.
One can rectify this issue by introducing the parallel propagator
$$K(P, Q): T_Q'M \to T_P'M,$$
which acts by parallel transport along a geodesic $\gamma$ from $Q$ to $P$ whenever $Q$ is contained in the cut locus of $P$.
Then, if $\Omega \ni P$ is contained in the cut locus of $P$, it is natural to generalize (\ref{intro excess}) to
\begin{equation}\label{intro excess 2}
\Exc_\Omega(U, P) := |\partial^* U \cap \Omega| - \left|\int_\Omega K(P, Q) \normal_U(Q) \dif \mu(Q)\right|.
\end{equation}
In order for (\ref{intro excess 2}) to be useful, however, it should be approximately independent of the basepoint $P$: mollification may move $\partial^* U$, so even if we choose a basepoint on $\partial^* U$, we may have to move it later if we wish to impose that $P \in \partial^* U$.

In order to prove that (\ref{intro excess 2}) is approximately translation-invariant, we must estimate $K(P, Q)$.
This is accomplished using the formula \cite[Chapter II, \S2]{baez1994gauge}
$$K(P, Q) = \mathcal Pe^{-\int_Q^P \Gamma}$$
where $\mathcal P$ is the path-ordering symbol, and $\Gamma$ is the Christoffel symbol of $M$, viewed as a matrix-valued $1$-form.
The Christoffel symbol $\Gamma$ is suitably small when measured in normal coordinates based at a point of $\Omega$.
This last point is cruical: the proof of the de Giorgi lemma does not allow us to use an arbitrary coordinate system on $M$, but only coordinate systems in which the linearization of the minimal surface equation around the trivial solution gives the Laplace equation.
However, we can choose normal coordinates in which this happens, and hence conclude the de Giorgi lemma.



\subsection{Acknowledgments}
I would like to thank Georgios Daskalopoulos for suggesting this project and for many helpful discussions, and Christine Breiner for helpful comments.
This research was supported by the National Science Foundation's Graduate Research Fellowship Program under Grant No. DGE-2040433.

\section{Preliminaries}\label{Prelims}
\subsection{Notation and conventions}
We denote the Riemannian metric by $g$.
When using the Einstein convention, Greek indices range over $0, 1, \dots$ while Latin indices range over $1, \dots$.
We write $y := x^0$.
We write $\Japan \xi := \sqrt{1 + |\xi|^2}$ for the Japanese norm of a vector $\xi$.
We write $\Ball^d$ for the unit ball in euclidean space, and $\Sph^{d - 1}$ for the unit sphere.

\subsection{Sets of locally finite perimeter}
We recall some elementary facts from geometric measure theory.
Standard references on euclidean space include \cite{simon1983GMT,Giusti77}, and we will only give details when the Riemannian case is significantly more involved than its euclidean counterpart.

Given a function $u \in L^1_\loc(M)$, we can define its \dfn{total variation measure} in an open set $U$ by
$$\|\dif u\|_{TV(U)} := \sup_{\|\psi\|_{C^0} \leq 1} \int_M u \dif \psi$$
where $\psi$ ranges over all $d-1$-forms with compact support in $U$.
We sometimes abuse notation and write $\star |\dif u|$ for the total variation measure, even if $|\dif u|$ is a distribution but not a function.
Similarly we sometimes write $\star \partial_i u$, if we have chosen a coordinate system $(\partial_i)$.

If the total variation measure is a Radon measure, we say that $u \in BV_\loc(M)$, or that $u$ has \dfn{locally bounded variation}.
This is a diffeomorphism-invariant condition.
If $u$ has locally bounded variation and $U \subseteq M$ is open with Lipschitz boundary, then by \cite[Theorem 2.10]{Miranda67}, the trace $u|_{\partial U} \in L^1_\loc(\partial U)$ is well-defined.
Moreover, by \cite[Theorem 4.14]{simon1983GMT}, there exists a $\star|\dif u|$-measurable section $f$, the \dfn{polar part} of $\dif u$, of the cosphere bundle $S'M$ such that $\partial_i u = f_i |\dif u|$.
As in \cite{Miranda66, Giusti77}, most of the technical work in this paper amounts to controlling the oscillation of a polar section $f$ at fine scales.
In order to make this precise, we shall need to take averages of $f$ and apply a version of the Lebesgue differentiation theorem for curved vector bundles.

\begin{proposition}[Lebesgue differentiation theorem for a vector bundle]\label{LebesgueDiff}
Let $E \to M$ be a vector bundle over an oriented smooth manifold $M$, $\omega$ a Radon measure, and $f \in L^1_\loc(M, E, \omega)$.
Then there exists an $\omega$-null set $Z \subset M$ such that for every Riemannian metric on $M$, every trivialization $(F_1, \dots, F_\ell)$ of $E$ with dual trivialization $(F'_1, \dots, F'_\ell)$ of $E'$, and every $P \in M \setminus Z$,
$$f(P) = \lim_{r \to 0} \sum_{i=1}^\ell \left[\frac{\int_{B(P, r)} (F'_i, f) \dif \omega}{\omega(B(P, r))}\right] F_i(P).$$
\end{proposition}

We shall apply this proposition with $E := T'M$, $F_i = \dif x^i$.
Note carefully that the terms inside the limit \emph{are} dependent on the metric and the choice of trivialization; the content of this result is that \emph{the set of Lebesgue points is diffeomorphism-invariant}.

\begin{proof}
Choose a flat Riemannian metric, let $\mathcal F = ((F_i), (F_i'))$ be a pair of paralellizations of $E, E'$ such that $(F_i', F_j) = \delta_{ij}$, and $\ell$ the rank of $E$.
Then for every $\delta > 0$ there exists $\tilde f \in C_c(M, E)$ such that $\|f - \tilde f\|_{L^1(\omega)} < \delta$, thus
\begin{align*}
&\left|\sum_{i=1}^\ell \left[(F_i'(x), f(x)) - \dashint_{B(x, r)} (F_i', f) \dif \omega\right] F_i(x)\right| \\
&\qquad \leq \left|\sum_{i=1}^\ell (F_i'(x), f(x) - \tilde f(x)) F_i(x)\right| + \dashint_{B(x, r)} \left|\sum_{i=1}^\ell (F_i', f - \tilde f)F_i(x) \dif \omega \right| \\
&\qquad \qquad + \left|\sum_{i=1}^\ell \left[(F_i'(x), \tilde f(x)) - \dashint_{B(x, r)} (F_i, \tilde f) \dif \omega\right] F_i(x)\right| \\
&\qquad =: I_1(x) + I_{2, r}(x) + I_{3, r}(x).
\end{align*}
Here the integral defining $I_{2, r}(x)$ is valued in the fiber $E_x$ and $I_{3, r} \to 0$, $\omega$-almost everywhere as $x \to 0$.

By the proof of the Lebesgue differentiation theorem \cite[Chapter 3, Theorem 1.3]{stein2009real},
\begin{align*}
&\left\{x \in M: \limsup_{r \to 0} \left|\sum_{i=1}^\ell \left[(F_i'(x), f(x)) - \dashint_{B(x, r)} (F_i', f) \dif \omega\right] F_i(x)\right| > 2\varepsilon\right\} \\
&\qquad \subseteq \{I_1 > \varepsilon\} \cup \bigcap_{r > 0} \bigcup_{s <r} \{I_{2, s} > \varepsilon\} \\
&\qquad \subseteq \{f - \tilde f| > \varepsilon\} \cup \bigcap_{r > 0} \bigcup_{s < r} \left\{\dashint_{B(x, r)} |f - \tilde f| \dif \omega > \varepsilon\right\}
\end{align*}
The right-hand side is independent of $\mathcal F$ and $\delta$, but has $\omega$-measure $\lesssim \delta/\varepsilon$.
We choose $\delta \ll \varepsilon$ to conclude that it is $\omega$-null, and remains as such if we take a union over all possible $\mathcal F$ and $\varepsilon$.
Moreover, if $h$ is any Riemannian metric, then its balls have bounded eccentricity with respect to $g$, so we may replace $g$-balls with $h$-balls in the above set inclusion \cite[Chapter 3, Corollary 1.7]{stein2009real}.
\end{proof}

\begin{corollary}
The polar section $f: M \to S'M$ of a $BV$ function $u$ satisfies
\begin{equation}\label{Lebesgue point definition}
    f_\mu(P) = \left[\lim_{r \to 0} \frac{\int_{B(P, r)} \partial_\mu u \dif V}{\|\dif u\|_{TV(B(P, r))}}\right]
\end{equation}
for any coordinate system $(x^\mu)$ and any Riemannian metric $g$, and $\star|\dif u|$-almost every $P$.
The exceptional set does not depend on $(x^\mu)$ or $g$.
\end{corollary}

It follows from the above corollary that the following definitions, largely taken from \cite[Definition 3.3]{Giusti77}, which a priori refer to the metric or to a choice of coordinate system, are actually completely determined by the smooth structure on $M$.

\begin{definition}
Let $U \subseteq M$. We say that $U$ has \dfn{locally finite perimeter} if $1_U \in BV_\loc(M)$.
In that case we make the following definitions:
\begin{enumerate}
\item The \dfn{measure-theoretic boundary} $\partial U$ is the set of points whose Lebesgue density with respect to $M$ is $\in (0, 1)$.
\item The polar section of $1_U$ is called the \dfn{conormal $1$-form} $\normal_U$ to $\partial U$.
\item The set of points $P$ for which $\normal_U(P)$ satisfies (\ref{Lebesgue point definition}) is the \dfn{reduced boundary} $\partial^* U$.
\item The \dfn{perimeter} $|\partial^* U \cap E|$ in a Borel set $E$ is $\|d1_U\|_{TV(E)}$.
\end{enumerate}
\end{definition}

Choosing a coordinate system on $M$ in which the volume form is $\dif x^0 \wedge \cdots \wedge \dif x^{d - 1}$ (so we may assume that $M$ is a domain in $\RR^d$), we see from \cite[Chapter 4]{Giusti77} that the following properties of the reduced boundary hold:

\begin{proposition}\label{locality of Caccioppoli}
    Let $U$ be a set of locally finite perimeter.
    Then:
    \begin{enumerate}
    \item $\partial^* U$ is a dense subset of $\partial U$.
    \item If $\normal_U$ extends to a continuous $1$-form on $\partial U$, then $\partial^* U = \partial U$ is a $C^1$ embedded hypersurface.
    \item If $\partial^* U = \partial U$ is a $C^1$ hypersurface, then $\normal_U$ is the conormal $1$-form on $\partial U$ as defined in differential topology, and the total variation measure of $1_U$ is the surface measure on $\partial U$.
\end{enumerate}
\end{proposition}


\begin{proposition}[coarea formula]\label{Coarea2}
Let $u \in BV_\loc(M)$ and $E$ an open set. Then
\begin{equation}\label{coarea formula}
\|\dif u\|_{TV(E)} = \int_{-\infty}^\infty |E \cap \partial^* \{u > y\}| \dif y.
\end{equation}
\end{proposition}
\begin{proof}
Reasoning identically to \cite[Theorem 1.23]{Giusti77}, we may assume that $u \in C^\infty(M)$.
If this is true and also $u$ has no critical points, then (\ref{coarea formula}) follows from Fubini's theorem, the fact that $|E \cap \partial \{u > y\}|$ is the surface area of $E \cap \{u = y\}$ (by Proposition \ref{locality of Caccioppoli}), and the change-of-variables formula.
However the left-hand side of (\ref{coarea formula}) is unaffected by critical points of $u$, and the right-hand side of (\ref{coarea formula}) is unaffected by critical values of $u$ by Sard's theorem, so (\ref{coarea formula}) holds even if $u \in C^\infty(M)$ has critical points.
\end{proof}

We write
$$\eta(u, U) := \inf_{v|_{\partial U} = 0} \|\dif(u + v)\|_{TV(U)}$$
for $u \in BV_\loc(M)$ and $U \subseteq M$ open with Lipschitz boundary, thus $u$ has least gradient iff $\eta(u, M) = \|\dif u\|_{TV(M)}$.
If $u, v \in BV(U)$, then using the coarea formula and reasoning analogously to \cite[Lemma 5.6]{Giusti77}, we obtain the a priori estimates
\begin{align}
|\eta(u, U) - \eta(v, U)| &\leq \|u - v\|_{L^1(\partial U)} \label{a priori estimate 1} \\
\eta(u, U) &\leq \|u\|_{L^1(\partial U)} \leq |\partial U| \cdot \|u\|_{L^\infty(M)}. \label{a priori estimate 2}
\end{align}

\subsection{The parallel propagator}
The key new ingredient for the proof of the regularity theorem on Riemannian manifolds is the parallel propagator
$$K(P, Q): T_Q'M \to T_P'M$$
which is defined whenever there exists a unique geodesic $\gamma$ from $Q$ to $P$.
It sends a cotangent vector to $M$ at $Q$ to its parallel transport along $\gamma$, which is a cotangent vector to $M$ at $P$.

To express the parallel propagator in coordinates, we recall from \cite[Chapter II, \S2]{baez1994gauge} that if $\Gamma$ is a square matrix of $1$-forms, its \dfn{path-ordered exponential} along a curve $\gamma: [0, 1] \to M$ is given by 
$$\mathcal Pe^{-\int_\gamma \Gamma} := \sum_{n=0}^\infty (-1)^n \int_{\Delta_n} \prod_{m=1}^n (\Gamma(\gamma(t_i)), \gamma'(t_i)) \dif t$$
where 
$$\Delta_n := \{t \in [0, 1]^n: t_1 \leq t_2 \leq \cdots \leq t_n\}$$
is the standard $n$-simplex.
The path-ordered exponential is then a square matrix; it is defined by a convergent series if $\Gamma$ is continuous \cite[Chapter II, \S2]{baez1994gauge}, and to first order, the path-ordered exponential is
\begin{equation}\label{path ordered exponential taylor series}
\mathcal Pe^{-\int_\gamma \Gamma} = I + O(|\gamma| \cdot \|\Gamma\|_{C^0}).
\end{equation}

\begin{proposition}
Suppose that there is a unique geodesic $\gamma$ from $Q$ to $P$.
Let $\Gamma$ be the Christoffel symbols of the Levi-Civita connection of $M$ acting on the cotangent bundle in some coordinate system, viewed as a $d \times d$-matrix of $1$-forms.
Then
\begin{equation}\label{path ordered exponential is propagator}
K(P, Q) = \mathcal Pe^{-\int_\gamma \Gamma}.
\end{equation}
\end{proposition}

To be more precise, if we use our coordinate system to identify $T_P'M$ and $T_Q'M$ with $\RR^d$, then the $d\times d$-matrix obtained from the path-ordered exponential can be viewed as a linear map
$$\mathcal Pe^{-\int_\gamma \Gamma}: T_Q'M \to T_P'M.$$
The assertion is that this map is exactly equal to $K(P, Q)$.
For a proof, see \cite[Chapter II, \S2]{baez1994gauge}.

We now use the parallel propagator to define an intrinsic averaging operator on the cotangent bundle.
Such an operator will necessarily depend on the choice of basepoint, since there is no canonical way to identify all of the cotangent spaces to $M$.

\begin{definition}
Let $\mu$ be a Radon measure, $U \subseteq M$ open with finite $\mu$-measure, $P \in U$, and $\xi$ a $1$-form on $U$.
Suppose that for every $Q \in U$ there exists a unique geodesic from $Q$ to $P$.
Then the \dfn{average} of $\xi$ in $U$ with respect to $\mu, P$ is 
$$\avg_{U, P, \mu} \xi := \frac{1}{\mu(U)} \int_U K(P, Q)\xi(Q) \dif \mu(Q).$$
\end{definition}

To be more explicit, $K(P, Q)\xi(Q)$ is a cotangent vector to $P$, so we have a vector-valued map 
\begin{align*}
U &\to T_PM \\
Q &\mapsto K(P, Q)\xi(Q).
\end{align*}
This vector-valued map has target a single vector space, rather than a vector bundle, so it makes sense to take its vector-valued integral.

\begin{proposition}\label{translation invariance}
Let $\mu$ be a Radon measure, $U \subseteq M$ open with finite $\mu$-measure, $P, Q \in U$, and $\xi$ a $1$-form on $U$.
Suppose also that $U$ has diameter $\leq \rho$ and for any $R \in U$ there exist unique geodesics from $R$ to $P$ and $Q$.
Then 
$$||\avg_{U, P, \mu} \xi| - |\avg_{U, Q, \mu} \xi|| \lesssim \rho^2 \|\xi\|_{L^\infty(\mu)}.$$
\end{proposition}
\begin{proof}
Choose normal coordinates at $P$, so that $\|\Gamma\|_{C^0(U)} \lesssim \rho$.
Moreover, the geodesics from $R$ to $P, Q$ have length $\leq \rho$.
Therefore by the reverse triangle inequality,
\begin{align*}
||\avg_{U, P, \mu} \xi| - |\avg_{U, Q, \mu} \xi||
&= \frac{1}{\mu(U)} \left|\left|\int_U \mathcal P e^{-\int_R^P \Gamma} \xi(R) \dif \mu(R)\right| - \left|\int_U \mathcal P e^{-\int_R^Q \Gamma} \xi(R) \dif \mu(R)\right|\right| \\
&\leq \frac{\|\xi\|_{L^\infty(\mu)}}{\mu(U)} \int_U \left|\mathcal P e^{-\int_R^P \Gamma} - \mathcal Pe^{-\int_R^Q \Gamma}\right| \dif \mu(R) \\
&\leq \|\xi\|_{L^\infty(\mu)} \sup_{R \in U} \left|\mathcal Pe^{-\int_R^P \Gamma} - \mathcal Pe^{-\int_R^Q \Gamma}\right|.
\end{align*}
By (\ref{path ordered exponential taylor series}), this quantity is $\lesssim \rho^2 \|\xi\|_{L^\infty(\mu)}$.
\end{proof}


\section{Monotonicity formula}\label{MollifierSection}
Functions of least gradient satisfy a monotonicity formula (see \cite[Theorem 5.12]{Giusti77} for a proof on euclidean space) which is essential to the proof of the de Giorgi--Miranda theorem in several ways.
Most prominently, the monotonicity formula implies the existence of tangent cones to sets of least perimeter (allowing us to apply the hypothesis $d \leq 7$), implies certain precise estimates on the perimeter of a set of least perimeter, and governs the mollification process used in the proof of the de Giorgi lemma.

To formulate the monotonicity formula, define for a function $u \in BV(M)$ and $P \in M$ the quantity
\begin{equation}\label{integral of du}
I(u, P, r) := \int_{B(P, r)} K(P, Q) \frac{\dif u}{|\dif u|}(Q) \dif \mu(Q)
\end{equation}
where $\dif u/|\dif u|$ is the polar section and $\dif \mu := \star|\dif u|$ is the total variation measure.
For the proof and corollaries of the monotonicity formula it will be useful to establish some notation: $\dif \sigma$ is the area form on the unit sphere $\Sph^{d - 1}$ and $B_r := B(P, r)$.
In normal polar coordinates $(r, \theta) \in \RR_+ \times \Sph^{d - 1}$, the area form on geodesic spheres takes the form 
\begin{equation}\label{geodesic spheres area form}
\dif S_{\partial B_\rho} = (\rho^{d - 1} + O(\rho^{d + 1})) \dif \sigma.
\end{equation}

\begin{proposition}[monotonicity formula]\label{Monotone}
There exist $A \geq 0$ and $r_* > 0$ such that for any function $u$ of least gradient on $M$, with total variation measure $\mu$, and any $0 < r \leq r_*$,
\begin{equation}\label{weak monotonicity}
\frac{\dif}{\dif r}\left[e^{Ar^2}r^{1 - d} \mu(B(P, r))\right] \geq 0.
\end{equation}
Moreover, for any $0 < r_1 < r_2 \leq r_*$,
\begin{align*}
&|r_2^{1 - d} I(u, P, r_2) - r_1^{1 - d} I(u, P, r_1)|^2 \\
&\qquad \lesssim \left(1 + (d - 1) \log \frac{r_2}{r_1}\right) \left(r_2^{1 - d}\mu(B(P, r_2)) \right)
\left(\int_{r_1}^{r_2} \partial_r \left[e^{Ar^2} r^{1 - d} \mu(B(P, r))\right] \dif r\right)\\
&\qquad \qquad + r_2^{6-2d} \left(\mu(B(P, r_2))\right)^2.
\end{align*}
\end{proposition}

\begin{lemma}\label{monotonicity lemma}
There exists $A$ such that for every $u \in C^1(B_R)$, $0 < r_1 < r_2 < R$, if we let
$$E(r) := \mu(B_r) - \eta(u, r),$$
so that $E(R) = 0$ iff $u$ has least gradient, then there exists $A \geq 0$ such that for $R > 0$ small,
\begin{equation}\label{monotonicity lemma eqn}
0 \leq \int_{B_{r_2} \setminus B_{r_1}} \star r^{1 - d}\frac{(\partial_ru)^2}{|\dif u|} \leq 2\int_{r_1}^{r_2} \partial_r \left[e^{Ar^2} r^{1-d}\int_{B_r} \star |\dif u|\right] + \frac{O(E(r))}{r^d} \dif r.
\end{equation}
\end{lemma}
\begin{proof}
We follow the proof of \cite[Lemma 5.8]{Giusti77}.
This result is coordinate-invariant, so we may use whichever coordinates are convenient: we in fact use normal polar coordinates $(r, \theta)$.
By (\ref{geodesic spheres area form}), there exists $A \geq 0$ such that $e^{A\rho^2} \sqrt{\det g|_{\partial B_\rho}}$ is increasing.

Fix $s \in [r_1, r_2]$ and introduce a competitor $v(r, \theta) = u(s, \theta)$.
From the definition of $\eta$,
$$\eta(u, s) \leq \int_U \star |\dif v| = \int_0^s \int_{\partial B_r} \star_r |\dif v| \dif r.$$
Since $\partial_r v = 0$,
$$\int_{\partial B_r} \star_r |\dif v| \leq e^{As^2} \frac{\tilde r^{d - 1}}{s^{d - 1}} \int_{\partial B_s} \star_s |\dif v|.$$
So by Fubini's theorem
\begin{align*}
\eta(u, s) &\leq e^{As^2} \int_0^s \frac{r^{d - 1}}{s^{d - 1}} \dif r \cdot \int_{\partial B_s} \star_s |\dif v| = \frac{s e^{As^2}}{d} \int_{\partial B_s} \star_s |\dif v|\\
&\leq \frac{s e^{As^2}}{d - 1} \int_{\partial B_s} \star_s |\dif v|.
\end{align*}
By Gauss' lemma, $\partial_r$ is orthogonal to $\partial B_s$, so $\dif v$ is the orthogonal projection of $\dif u$ onto $T' \partial B_s$, and its orthocomplement is $\partial_r u$. Therefore by Taylor's theorem,
$$\int_{\partial B_s} \star_s |\dif v| \leq \int_{\partial B_s} \star_s |\dif u| \sqrt{1 - \frac{(\partial_r u)^2}{|\dif u|^2}} \leq \int_{\partial B_s} \star_s \left[|\dif u| - \frac{(\partial_r u)^2}{2 |\dif u|}\right]$$
or in other words
\begin{align*}
\int_{\partial B_s} \star_s \frac{(\partial_r u)^2}{2|\dif u|} &\leq \int_{\partial B_s} \star_s |\dif u| - \frac{d - 1}{s} e^{-As^2} \eta(u, s)\\
&\leq \int_{\partial B_s} \star_s |\dif u| - \frac{d - 1}{s} e^{-As^2} \int_{B_s} \star |\dif u| - O(s^{-1}E(s)).
\end{align*}
We moreover have for $\tilde A \geq 0$ that
$$e^{-\tilde As^2} \partial_s \left[e^{\tilde As^2} s^{1 - d} \int_{B_s} \star |\dif u|\right] = \left[2\tilde As^{2 - d} - \frac{d - 1}{s^d}\right]\int_{B_s} \star |\dif u| + s^{1 - d} \int_{\partial B_s} \star_s |\dif u|$$
so if we choose $\tilde A$ so that
$$-\frac{d - 1}{s} e^{-As^2} = 2\tilde As - \frac{d - 1}{s}$$
then
$$s^{1 - d} \int_{\partial B_s} \star_s |\dif u| - (d - 1)\frac{e^{-\tilde As^2}}{s^d} \int_{B_s} \star|\dif u| \leq e^{-\tilde As^2} \partial_s\left(e^{\tilde As^2} s^{1 - d} \int_{B_s} \star|\dif u|\right).$$
We moreover have $e^{-\tilde As^2} \leq 1$, so we can now integrate with respect to $\dif s$ and rename $\tilde A$ to $A$ to conclude.
\end{proof}

\begin{proof}[Proof of Proposition \ref{Monotone}]
Working in normal coordinates $(x^\alpha)$ based at $P$, we first computing using (\ref{path ordered exponential taylor series})
$$I_r := r^{1 - d} I(u, P, r) = r^{1 - d} \left[\int_{B(P, r)} \star \partial_\alpha u\right] \dif x^\alpha(P) + O(r^{3 - d}) \|\dif u\|_{TV(B(P, r))}.$$
Let $(\iota_0, \dots, \iota_{d - 1}): \Sph^{d - 1} \to \RR^d$ be the embedding of the unit sphere.
Applying the fundamental theorem of calculus,
\begin{align*}
(I_r)_\alpha &= \left[\int_{\Sph^{d - 1}} u(r, \theta) \iota_\alpha(\theta) \dif \sigma(\theta)\right] + O(r^{3 - d}) \|\dif u\|_{TV(B(P, r))}
\end{align*}
where we identified $\partial B_r$, the domain of $u(r, \cdot)$, with $\Sph^{d - 1}$ using normal coordinates.
Therefore
\begin{equation}\label{monotone dump the metric}
|I_{r_2} - I_{r_1}| \leq \int_{\Sph^{d - 1}} |u(r_2, \theta) - u(r_1, \theta)| \dif \sigma(\theta) + O(r^{3 - d}) \|\dif u\|_{TV(B(P, r))}.
\end{equation}
The metric $g$ plays no role in the dominant term of (\ref{monotone dump the metric}), so we may use \cite[Lemma 5.3]{Giusti77} to bound
$$0 \leq \int_{\Sph^{d - 1}} |u(r_2, \theta) - u(r_1, \theta)| \dif \sigma(\theta) \leq \int_{\Sph^{d - 1}} \int_{r_1}^{r_2} r^{1 - d}|\partial_r u(r, \theta)| \dif r \dif\sigma(\theta).$$
To reintroduce the metric we posit that $r_2$ is small enough that $\dif r \dif \sigma(\theta) \leq \star 2$.
We therefore have
\begin{equation}\label{monotone before cs}
\int_{\Sph^{d - 1}} \int_{r_1}^{r_2} r^{1 - d}|\partial_r u(r, \theta)| \dif r \dif\sigma(\theta) \leq 2 \int_{B_{r_2} \setminus B_{r_1}} \star r^{1 - d}|\partial_r u|
\end{equation}
and if we apply the Cauchy-Schwarz inequality and approximate $u$ by $C^1$ functions as on \cite[pg68]{Giusti77}, it follows from Lemma \ref{monotonicity lemma} that the right-hand side of (\ref{monotone before cs}) is
$$\lesssim \sqrt{\int_{B_{r_2} \setminus B_{r_1}} \star r^{1 - d} |\dif u|} \sqrt{\int_{r_1}^{r_2} \partial_r \left[e^{Ar^2} r^{1-d}\int_{B_r} \star |\dif u|\right] \dif r}.$$
The monotonicity (\ref{weak monotonicity}) follows at once.

Integrating by parts,
\begin{align*}
\int_{B_{r_2} \setminus B_{r_1}} \star r^{1 - d} |\dif u| &= \int_{r_1}^{r_2} r^{1 - d} \partial_r \int_{B_r} \star |\dif u| \dif r \\
&\leq r^{1 - d} \int_{B_r} \star |\dif u| + (d - 1) \int_{r_1}^{r_2} r^{-d} \int_{B_r} \star |\dif u| \dif r.
\end{align*}
Using (\ref{weak monotonicity}) we bound this second integral as
\begin{align*}
\int_{r_1}^{r_2} r^{-d} \int_{B_r} \star |\dif u| \dif r &\leq r^{1 - d} \log \frac{r_2}{r_1} \int_{B_{r_2}} \star |\dif u|.
\end{align*}
If we set
$$J_r := r^{1 - d} \int_{B_r} \star |\dif u|$$
then we can sum up our progress so far as
$$|I_{r_2} - I_{r_1}| \lesssim \sqrt{\left[1 + \log \frac{r_2}{r_1}\right] J_{r_2}} \sqrt{e^{Ar_2^2} J_{r_2} - e^{Ar_1^2} J_{r_1}} + r_2^2 J_{r_2}.$$
The claim now follows by squaring both sides and applying Cauchy-Schwarz.
\end{proof}

For a function $u$ on $M$, $P \in M$, we define the \dfn{blowup} of $u$ at $P$ to be the net of functions $u_t: T_PM \to \RR$, given by
$$u_t(v) := u\left(\exp_P(tv)\right).$$
If $u$ has least gradient and $M$ is curved, then (unlike in the flat case) its blowup $u_t$ need not have least gradient, but at least satisfies 
\begin{equation}\label{approximately least gradient}
\limsup_{t \to 0} \|\dif u_t\|_{TV(U)} \leq \limsup_{t \to 0} \eta(u_t, U) < \infty 
\end{equation}
where $U$ is any open subset of $T_PM$ with Lipschitz boundary.
This can be easily seen by Taylor expanding the metric in normal coordinates based at $P$; as the scale parameter $t \to 0$, the corrections from the metric vanish.
The Miranda stability theorem \cite[Osservazione 3]{Miranda66}, which traditionally has been stated for sequences of functions of least gradient, can be easily modified to hold in this setting.
We omit the details.

\begin{lemma}[Miranda stability theorem]\label{Miranda convergence}
If a net of functions $(u_t)$ (not necessarily of the same trace) satisfies (\ref{approximately least gradient}) for every open $U \Subset M$ with Lipschitz boundary,
then a subsequence converges weakly in $BV_\loc(M)$ to some function $u$ of least gradient.
\end{lemma}

\begin{corollary}\label{blowup theorem}
Suppose that $U$ is an open set with least perimeter in $B(P, r)$, $P \in \partial^* U$, and $u = 1_U$.
Then the blowup $(u_t)$ of $u$ converges as $t \to 0$ along a subsequence weakly in $BV$ to the indicator function of a set $C \subset T_PM$ such that $\partial C$ is a minimal cone containing $0$.
If $d \leq 7$, then $\partial C$ is a hyperplane.
\end{corollary}
\begin{proof}
By the Miranda stability theorem applied on the manifold $T_PM$, $u_t \to 1_C$ weakly in $BV$, where $C$ has least perimeter.
It is a standard consequence of (\ref{weak monotonicity}) that, since $(u_t)$ is a blowup, $\partial C$ is a minimal cone containing $0$.
See \cite[Theorem 9.3]{Giusti77} for the euclidean case; the proof is identical here.
If $d \leq 7$, it follows that $\partial C$ is a hyperplane \cite[Theorem 9.10 and Theorem 10.10]{Giusti77}.
\end{proof}

We now estimate the perimeter of a set of least perimeter, generalizing \cite[Remark 5.13]{Giusti77}.
As a consequence, if $u = 1_U$ where $U$ has least perimeter, then the error term in the monotonicity formula is of size $O(r_2^{d + 1})$.

\begin{corollary}\label{doubling dimension}
There exists $A \geq 0$ such that for every set $U$ of least perimeter in a ball $B_r = B(P, r)$, with $P \in \partial^* U$, and $r > 0$ small,
$$|\Ball^{d - 1}|e^{-Ar^2}r^{d - 1} \leq |\partial^*U \cap B_r| \leq |\Sph^{d - 1}|e^{Ar^2} r^{d - 1}.$$
\end{corollary}
\begin{proof}
The upper bound on $|\partial^* U \cap B_r|$ is immediate from (\ref{a priori estimate 2}) and (\ref{geodesic spheres area form}).
The lower bound is obtained from (\ref{weak monotonicity}), which implies that
$$\limsup_{\rho \to 0} e^{-A\rho^2} \rho^{1 - d} |\partial^* U \cap B_\rho| \leq |\partial^* U \cap B_r|.$$
To control the left-hand side we take a blowup $(u_\rho)$ of $1_U$.
By Corollary \ref{blowup theorem} we can pass to a subsequence so that $u_\rho \to 1_C$ for $C$ a half-space, which in particular is transverse to $B'_1$, where the prime denotes the euclidean metric on the tangent space.
Then
\begin{align*}
\limsup_{\rho \to 0} e^{-A\rho^2} \rho^{1 - d} |\partial^* U \cap B_\rho| &= \lim_{\rho \to 0} e^{O(\rho^2)} \int_{B'_1} \star'|\dif u_\rho|' = \int_{B'_1} \star'|\dif 1_C|.
\end{align*}
This last term is $|\partial C \cap B'_1|$, the measure of the intersection of the euclidean unit ball with a minimal cone through its origin.
By \cite[(5.16)]{Giusti77}, any minimal cone in $B'_1$ has measure at least $|\Ball^{d - 1}|$, the area of the unit ball of $\RR^{d - 1}$.
\end{proof}

\section{De Giorgi lemma}\label{DeGiorgiSec}
We now introduce the quantity which governs the rate of convergence of the Lebesgue differentiation theorem for $\normal_U$, whenever $U$ is a set of locally finite perimeter.
More precisely, let $\mu$ be the surface measure of $\partial^* U$, which is by definition the total variation measure of $1_U$.
We study the convergence of the approximation
$$\normal_U(P, r) := \avg_{B(P, r), P, \mu} \normal_U.$$

\begin{definition}
The \dfn{excess} of a set $U \subset M$ of locally finite perimeter at $P \in \partial U$, such that $\partial^* U$ has surface measure $\mu$, in an open set $A \ni P$ with Lipschitz boundary is
$$\Exc_A(U, P) := \mu(U)\left(1 - \left|\avg_{A, P, \mu} \normal_U\right|\right).$$
For $\rho > 0$ we write $\Exc_\rho(U, P) := \Exc_{B(P, \rho)}(U, P)$.
\end{definition}

Since parallel transport preserves length, we have
$$|K(P, Q) \normal_U(Q)| = 1$$
and, averaging $K(P, Q) \normal_U(Q)$ in $Q$, we conclude
\begin{equation}\label{normal isnt too big}
|\normal_U(P, r)| \leq 1.
\end{equation}
From this we obtain the following monotonicity property: if $P \in A'$ and $A' \subseteq A$, then
\begin{equation}\label{approximate monotone}
0 \leq \Exc_{A'}(U, P) \leq \Exc_A(U, P).
\end{equation}
Moreover, by Proposition \ref{translation invariance}, if $P, Q \in A$ then 
\begin{equation}\label{translation invariant excess}
|\Exc_A(U, P) - \Exc_A(U, Q)| \lesssim (\diam A)^2 |\partial^* U \cap A|.
\end{equation}

Following \cite{Miranda66,Giusti77,deGiorgi61}, we proceed by controlling the excess using the following de Giorgi lemma, which is the Riemannian generalization of \cite[Theorem 8.1]{Giusti77}:

\begin{proposition}[de Giorgi lemma]\label{de Giorgi}
There exist constants $C, c, \rho_* > 0$ which only depend on $M$, such that for every $P \in M$, every $\rho$ such that $0 < \rho < \rho_*$, and every set $U \subset M$ of least perimeter such that
\begin{equation}\label{base case}
\Exc_\rho(U, P) \leq c\rho^{d - 1},
\end{equation}
we have
\begin{equation}\label{dGL concl}
\Exc_{\rho/2}(U, P) \leq 2^{-d} \Exc_\rho(U, P) + C\rho^{d + 1}.
\end{equation}
\end{proposition}

\begin{corollary}\label{DGL base case}
Let $d \leq 7$.
Assume the de Giorgi lemma, and let $U \subset M$ have least perimeter.
Then there exists $\rho = \rho(P) < \rho_*$, which is locally uniformly positive, such that (\ref{base case}) holds.
\end{corollary}
\begin{proof}
We follow \cite[pg109]{Giusti77}.
Let $Q \in \partial^* U$; we shall choose $\rho$ uniformly in a small neighborhood of $Q$, which is enough since $\partial^* U$ is dense in $\partial U$.
Since $\partial U$ has a tangent space at $Q$ by Corollary \ref{blowup theorem} and the fact that $d \leq 7$, $\Exc_r(U, Q) \ll r^{d - 1}$, thus we can choose $r \in (0, \rho_*)$ such that $\Exc_r(U, Q) \leq c(r/2)^{d - 1}$.
For $P \in B(Q, r/4)$ we have $B(P, r/4) \subseteq B(Q, r/2)$ and hence by the de Giorgi lemma, for $\rho := r/4$, (\ref{base case}) holds.
\end{proof}

\begin{proof}[Proof of Theorem \ref{main thm}, assuming the de Giorgi lemma]
We shall use the de Giorgi lemma inductively as in \cite[Theorem 8.2]{Giusti77}.
Let $P \in M$, $\rho$ be given by Corollary \ref{DGL base case}, $r := \rho/2^n$ for some $n \in \NN$, and $s \in (r/2, r)$.
Let
$\xi := \normal_U(P, r)$, $\eta = \normal_U(P, s)$, $m := |\partial^* U \cap B(P, s)|$, $M := |\partial^* U \cap B(P, r)|$, and $\gamma_n := \Exc_{\rho/2^n}(U, P)$.
Then $|\xi|^2 \leq 1$ and $|\eta|^2 \leq 1$, so
$$|\xi - \eta|^2 = |\xi|^2 + |\eta|^2 - 2 g^{-1}(\xi, \eta) \leq 2(1 - g^{-1}(\xi, \eta)).$$
To estimate the right-hand side, let $\mu$ be the surface measure on $\partial^* U$. Then
$$m(1 - g^{-1}(\xi, \eta)) = \int_{B(P, s)} (1 - g^{-1}(\xi, K(P, Q) \normal_U(Q))) \dif \mu(Q) =: I.$$
By the Cauchy-Schwarz inequality, (\ref{normal isnt too big}), and the fact that $K(P, Q)$ preserves length,
$$g^{-1}(\xi, K(P, Q) \normal_U(Q)) \leq 1,$$
which implies that, since $s \leq r$,
\begin{align*}
I
&\leq \int_{B(P, r)} (1 - g^{-1}(\xi, K(P, Q) \normal_U(Q))) \dif \mu(Q) 
\leq M(1 - |\xi|^2) \leq 2M(1 - |\xi|).
\end{align*}
But $M(1 - |\xi|)$ is exactly the definition of $\Exc_r(U, P)$, so putting everything together and using Corollary \ref{doubling dimension} to bound $m \gtrsim s^{d - 1} \gtrsim r^{d - 1}$, we have the bound
\begin{equation}\label{just need the excess}
|\xi - \eta|^2 \lesssim r^{1 - d} \Exc_r(U, P) = r^{1 - d} \gamma_n.
\end{equation}

By the de Giorgi lemma, there exists $C > 0$ such that
\begin{equation}\label{induction on gamma}
\gamma_n \leq \frac{\gamma_0}{2^{nd}} + \sum_{k=0}^n \frac{C\rho^{d + 1}}{2^{k(d + 1) + (n - k)d}} \leq \frac{\gamma_0 + C\rho^{d + 1}}{2^{nd}}.
\end{equation}
Indeed, by induction,
\begin{align*}
\gamma_{n + 1}
&\leq \frac{\gamma_0}{2^{(n + 1)d}} + 2^{-d} \sum_{k=0}^n \frac{C\rho^{d + 1}}{2^{k(d + 1) + (n - k)d}} + \frac{C\rho^{d + 1}}{2^{(n + 1)(d + 1)}} \\
&= \frac{\gamma_0}{2^{(n + 1)d}} + \sum_{k=0}^{n + 1} \frac{C\rho^{d + 1}}{2^{k(d + 1) + (n + 1 - k)d}}
\end{align*}
and (\ref{induction on gamma}) follows by summing the geometric series.
By (\ref{just need the excess}), (\ref{induction on gamma}), and the fact that $\rho$ is locally uniformly positive, it follows that along a subsequence, $\normal_U(\cdot, r)$ converges locally uniformly to $\normal_U$.
Therefore $\normal_U$ is continuous, so by Proposition \ref{locality of Caccioppoli}, $\partial U$ is a $C^1$ embedded hypersurface.
By elliptic regularity, $\partial U$ is smooth.

To prove the stability, let $(N_t)$ be a normal variation of $\partial U$ with compact support in the interior.
Then for $t$ small, $N_t$ bounds an open set $U_t$, and by the convexity (\ref{least perimeter dfn}) of the total variation,
$$|N_t| = \|\dif 1_{U_t}\|_{TV} \geq \|\dif 1_U\|_{TV} = |N_0|.$$
The stability then follows from the Taylor expansion of $|N_t|$ about $t = 0$.
\end{proof}

We break the proof of the the de Giorgi lemma into two steps, which correspond to \cite[Lemma 6.4]{Giusti77} (the $C^1$ case) and \cite[Lemma 7.5]{Giusti77} (reduction to the $C^1$ case).
Let $c_0 = c_0(M) > 0$ be a small constant to be determined later, and let $\alpha := 1/(2(1 - c_0))$.
We say that a vector field $X$ is $P$-\dfn{aligned} if there exists a normal coordinate frame $(\partial_\mu)$ based at $P$ such that $X = \partial_0$.

\begin{lemma}[de Giorgi lemma, $C^1$ case]\label{Miranda44}
Suppose that $c_0$ is small enough depending on $M$.
Then there exists a constant $c_1 = c_1(c_0, M)$, such that the following holds.
For every $\rho, \beta > 0$ small enough depending on $c_0$, and every set $U$ with $C^1$ boundary in $B(P, \rho)$, suppose that there exists a $P$-aligned vector field $X$ such that on $B(P, \rho)$,
\begin{align}
\Exc_\rho(U, P) &\leq \beta, \label{Miranda44 induction hyp}\\
|\partial^* U \cap B(P, \rho)| &\leq \eta(U, B(P, \rho)) + c_1 \beta, \label{Miranda44 minimality hyp} \\
(\normal_U, X) &\geq e^{-o(c_1^2)}. \label{Miranda44 normal hyp}
\end{align}
Then
\begin{equation}\label{Miranda44 concl}
\Exc_{\alpha \rho} (U, P) \leq (\alpha^{d + 1} + c_0) \beta + C\rho^{d + 1}.
\end{equation}
\end{lemma}

\begin{lemma}[reduction to the $C^1$ case]\label{single mollify}
There exists $c_2(c_1, M) > 0$ such that for any ball $B(P, \rho)$, with $\rho$ sufficiently small depending on $c_1$, and any set $U$ of least perimeter in $B(P, \rho)$ such that
\begin{equation}\label{single mollify hyp}
\Exc_\rho (U, P) \leq c_2 \rho^{d - 1},
\end{equation}
there exists an open set $V \subseteq B(P, (1 - c_1)\rho)$ and a $P$-aligned vector field $X$ such that $\partial V$ is $C^1$ and on $B(P, \rho)$, and such that for any $0 < \varpi \leq (1 - c_1)\rho$,
\begin{align}
|\Exc_\varpi (U, P) - \Exc_\varpi (V, P)| &\leq c_1 \Exc_\rho (U, P) + C\rho^{d + 1}, \label{single mollify excess} \\
|\partial V \cap B(P, (1 - c_1)\rho)| &\leq \eta(V, B(P, (1 - c_1)\rho)) + c_1 \Exc_\rho (U, P), \label{single mollify minimality} \\
(\normal_V, X) &\geq e^{-o(c_1^2)}. \label{single mollify normal}
\end{align}
\end{lemma}

In these lemmata, (\ref{Miranda44 induction hyp}), should be viewed as a bootstrap hypothesis, which is then improved by (\ref{Miranda44 concl}), while (\ref{Miranda44 minimality hyp}), (\ref{single mollify minimality}) asserts that the set in question is close to a set of least perimeter.
The technical condition (\ref{Miranda44 normal hyp}), (\ref{single mollify normal}) is used to find coordinates in which $\partial U$ can be represented by a graph; unlike (\ref{Miranda44 induction hyp}) and (\ref{Miranda44 minimality hyp}), which are conditions on the average behavior of $\partial U$, (\ref{Miranda44 normal hyp}) is a pointwise condition.

\begin{proof}[Proof of the de Giorgi lemma, assuming Lemmata \ref{Miranda44} and \ref{single mollify}]
Let $c_0 \leq 2^{-(d + 2)}$ and $c_1 \leq c_0/4$ be as in Lemma \ref{Miranda44}, and let $c := c_2$ be as in Lemma \ref{single mollify}.
Let $U$ be a set of least perimeter satisfying (\ref{single mollify hyp}), and let $V, X$ be as in Lemma \ref{single mollify}.
By (\ref{approximate monotone}) and (\ref{single mollify excess}),
\begin{align*}
\Exc_{(1 - c_1) \rho} (V, P) &\leq \Exc_{(1 - c_1) \rho} (U, P) + c_1 \Exc_\rho (U, P) + O(\rho^{d + 1}) \\
&\leq (1 + c_1) \Exc_\rho (U, P) + O(\rho^{d + 1}).
\end{align*}
Since $1/2 = \alpha (1 - c_0)$, by Lemma \ref{Miranda44}, it follows that
\begin{align*}
    \Exc_{\rho/2} (V, P) &\leq (2^{-(d + 1)} + c_0) (1 + c_1) \Exc_\rho (U, P) + O(\rho^{d + 1}) \\
    &\leq (2^{-(d + 1)} + 2^{-(d + 2)}) \Exc_\rho (U, P) + O(\rho^{d + 1})
\end{align*}
Finally, by (\ref{approximate monotone}) and (\ref{single mollify excess}),
\begin{align*}
    \Exc_{\rho/2} (U, P)
    &\leq \Exc_{\rho/2} (V, P) + c_1 \Exc_\rho (U, P) + O(\rho^{d + 1}) \\
    &\leq (2^{-(d + 1)} + 2^{-(d + 2)} + 2^{-(d + 3)}) \Exc_\rho (U, P) + O(\rho^{d + 1})\\
    &\leq 2^{-d} \Exc_\rho (U, P) + O(\rho^{d + 1}). \qedhere
\end{align*}
\end{proof}

\section{The smooth case}\label{smoothcase}
\subsection{Minimal graphs}
We now prove Lemma \ref{Miranda44}, following \cite[Chapter 6]{Giusti77}.
To set up the proof, let $(x^0, \dots, x^{d - 1})$ be normal coordinates based at $P$ such that $X = \partial_0$, and let $\Omega \subseteq \{x^0 = 0\}$ be a relatively open set.
We say that a $C^1$ hypersurface $N \subset M$ is \dfn{graphical} over $\Omega$ if for every integral curve $\gamma$ of $X$ passing through $\Omega$, $N$ intersects exactly once, and transversely.\footnote{Elsewhere in the literature this condition is called \dfn{strongly graphical}, and \dfn{graphical} means that the intersection may fail to be transverse.}
In coordinates, this means that $N = \{x^0 = u(x_1, \dots, x_{d - 1})\}$ for some $C^1$ \dfn{graphical function} $u: \Omega \to \RR$.
We equip $\Omega$ with the flat metric arising from the coordinates $(x^1, \dots, x^{d - 1})$, and let $\dif x$ denote the area form on $\Omega$ arising from its flat geometry.
The condition of being graphical only depends on $\Omega, X$, but not on the choice of coordinates.

\begin{lemma}
There is a map $\Lagrange$ taking $1$-jets on $\Omega$ to area forms on $\Omega$, with the following property: for any graphical hypersurface $N$ over $\Omega$, with graphical function $u$, and any $\Omega' \subseteq \Omega$, the area of the subset $N'$ of $N$ which is graphical over $\Omega'$ is
$$|N'| = \int_{\Omega'} \Lagrange(u, \dif u).$$
Moreover, if $\|\dif u\|_{C^0} \leq 1$, then
$$\Lagrange(u, \dif u) = \sqrt{1 + |\dif u|^2 + O(|x|^2 + \|u\|_{C^0}^2)} \dif x.$$
\end{lemma}
\begin{proof}
Let $\Psi: \Omega \to N$ be the diffeomorphism $x \mapsto (x, u(x))$.
For any $v, w \in T_x \Omega$, let
\begin{align*}
h(v, w) &:= g(x, u(x))(v, w), \\
h(X, v) &:= g(x, u(x))(X(x, u(x)), v), \\
h(X, X) &:= g(x, u(x))(X(x, u(x)), X(x, u(x))).
\end{align*}
These quantities are defined in coordinates, because we can use the coordinates to identify $T_{(x, 0)}M$ with $T_{(x, u(x))} M$, so that the quadratic form $g(x, u(x))$ on $T_{(x, u(x))} M$ acts on $T_{(x, 0)}M$.
In these coordinates we have $\Psi_* v = v + (\partial_v u) X(x, u(x))$, so that
\begin{align*}
\Psi^* g(v, w)
&= g(\Psi_* v, \Psi_* w) \\
&= h(v, w) + h(X, v) \partial_w u + h(X, w) \partial_v u + h(X, X) \partial_v u \partial_w u.
\end{align*}
If we set $v = \partial_i$ and $w = \partial_j$, $i, j = 1, \dots, d - 1$, then from the fact that the coordinates are normal we obtain $h(v, w) = \delta_{ij} + O(|x|^2 + u(x)^2)$, $h(X, v) = O(|x|^2 + u(x)^2)$, and $h(X, X) = 1 + O(|x|^2 + u(x)^2)$.
Moreover, $\partial_i u \partial_j u$ are the components of $\dif u \otimes \dif u$.
In conclusion,
$$\Psi^* g(\partial_i, \partial_j) = I + \dif u \otimes \dif u + O((|x|^2 + u(x)^2)(1 + |\dif u(x)|)).$$
However, $1 + |\dif u(x)| \leq 2$, which can be absorbed.
Furthermore, by \cite[(24)]{Petersen2008}, the determinant of $I + \dif u \otimes \dif u$ is $1 + (|\dif u|')^2$, which up to an error of size $|x|^2 + u(x)^2$ is equal to $|\dif u|^2$, since we are in normal coordinates.
The pullback of the area form of $N$ by $\Psi$ is $\sqrt{\det((\Psi^* g(\partial_i, \partial_j))_{ij})}$, and $\Psi$ identifies $N'$ with $\Omega'$, so we get 
\begin{align*}
|N'| &= \int_{\Omega'} \sqrt{1 + |\dif u(x)|^2 + O(|x|^2 + \|u\|_{C^0}^2)} \dif x. \qedhere
\end{align*}
\end{proof}

Let us write $\mathscr B_\rho$ for a ball in $\Omega$, of radius $\rho$ centered on the coordinate origin $P$.
If the graphical hypersurface $N$ intersects $\Omega$ at $P$, then on $\mathscr B_\rho$, we have from the Taylor expansion of $g$ that
\begin{equation}\label{approximate by euclidean lagrangian}
\Lagrange(u, \dif u) - \Lagrange(u, \avg_\rho \dif u) = \left(\Japan{\dif u} - \Japan{\avg_\rho \dif u}\right) \dif x + O(\rho^2)
\end{equation}
where $\avg_\rho := \avg_{\mathscr B_\rho, P, \dif x}$, taken using the flat metric on $\Omega$.
Here $\Japan{\dif u} := \sqrt{1 + |\dif u|^2}$ is the Japanese norm of $\dif u$, so that $\Japan{\dif u} \dif x$ is exactly the Lagrangian density for the euclidean minimal surface equation.
Thus, we reduce the problem of estimating $\Lagrange(u, \dif u) - \Lagrange(u, \avg_\rho \dif u)$ to the euclidean case, and hence can show the following analogue of \cite[Lemma 6.3]{Giusti77}.

\begin{lemma}[de Giorgi lemma, minimal graphs]\label{Miranda43}
There exists $c_1 = c_1(c_0, M) > 0$ such that for every $\beta, \rho > 0$, the following holds.
Let $N$ be a $C^1$ graphical hypersurface over $\mathscr B_\rho$, with graphical function $u$, which intersects $\Omega$ at $P$.
Let $I$ be the union of all integral curves of $X$ through $\mathscr B_\rho$.
Suppose that $\|\dif u\|_{C^1} \leq c_1$, and that
\begin{align}
\int_{\mathscr B_\rho} \Lagrange(u, \dif u) - \Lagrange(u, \avg_\rho \dif u) &\leq \beta \label{Miranda43 oscillation}, \\
\int_{\mathscr B_\rho} \Lagrange(u, \dif u) &\leq \eta(N, I) + c_1 \beta \label{Miranda43 minimality}.
\end{align}
Then for any $\tilde \alpha = \frac{1}{2} + O(c_0)$,
\begin{equation}\label{Miranda43 concl}
\int_{\mathscr B_{\tilde \alpha \rho}} \Lagrange(u, \dif u) - \Lagrange(u, \avg_{\tilde \alpha \rho} \dif u) \leq \left(\tilde \alpha^{d + 1} + \frac{c_0}{2}\right) \beta + C\rho^{d + 1}.
\end{equation}
\end{lemma}
\begin{proof}
Let $v$ be the harmonic function on $\mathscr B_\rho$ (with its euclidean metric) which equals $u$ on $\partial \mathscr B_\rho$.
By elliptic regularity, the maximum principle for harmonic functions, and (\ref{Miranda43 minimality}),
$$\|v\|_{C^1} \lesssim \|v\|_{C^0} \leq \|u\|_{C^0} \leq \rho \|\dif u\|_{C^1} \leq c_1.$$
In particular, $\Japan{\dif v} \lesssim 1$ and $v(x)^2 \lesssim \rho^2$, so by (\ref{approximate by euclidean lagrangian}) and the fact that $|\mathscr B_\rho| \sim \rho^{d - 1}$,
\begin{align*}
&\left|\int_{\mathscr B_\rho} \Lagrange(u, \dif u) - \Lagrange(v, \dif v) - \Japan{\dif u} \dif x + \Japan{\dif v} \dif x\right| \\
&\qquad \lesssim \int_{\mathscr B_\rho} (|x|^2 + u(x)^2) \Japan{\dif u} \dif x + (|x|^2 + v(x)^2) \Japan{\dif v} \dif x
\lesssim \rho^{d + 1}.
\end{align*}
Let $K$ be the graph of $v$, viewed as a submanifold of $M$.
Since $u$ and $v$ have the same trace, $|K| \geq \eta(K, I) = \eta(N, I)$, so
\begin{align*}
\int_{\mathscr B_\rho} \Lagrange(u, \dif u) - \Lagrange(v, \dif v) &\leq \int_{\mathscr B_\rho} \Lagrange(u, \dif u) - \eta(\Psi_w(\Omega), I) \leq c_1 \beta.
\end{align*}
We can then replace $\beta$ with $\beta + O(\rho^{d + 1})$ so that $u, v$ meet the hypotheses of \cite[Lemma 6.2]{Giusti77} which gives the result if $c_1$ is small enough.
\end{proof}

\subsection{Reduction to minimal graphs}
We have thus established the $C^1$ de Giorgi lemma, under the additional assumptions that $P \in N$, and that $N$ is graphical with a graphical function $u$ such that $\|\dif u\|_{C^0} \leq c_1$.
We shall deal with the assumption that $P \in N$ later, but first we remove the assumption that $N$ is graphical with small derivative.

\begin{lemma}\label{hypersurfaces are graphical}
Suppose that $\rho$ is small enough depending on $c_1$.
Let $U$ be a set with $C^1$ boundary in $B(P, \rho)$.
Suppose that $X$ is an aligned vector field for some normal coordinates $(x^\mu)$ based at $Q \in B(P, \rho)$, such that on $B(P, \rho)$,
\begin{align}
(\normal_U, X) &\geq e^{-o(c_1^2)}. \label{rep as a good graph hyp}
\end{align}
Then there exists a relatively open set $\Omega \subseteq \{x^0 = 0\}$ such that $\partial U$ is $X$-graphical over $\Omega$, and its graphical function $u$ satisfies $\|\dif u\|_{C^0} \leq c_1$.

Moreover, there exists a ball $\mathscr B \subseteq \Omega$ with the following property.
Let $\Omega^\alpha$ be the set of points in $\Omega$ which are initial data for $X$-integral curves which intersect $\partial U \cap B(P, \alpha \rho)$.
Then 
\begin{equation}\label{rep as a good graph set nests}
    \Omega^\alpha \subseteq \left(\alpha + \frac{c_0}{2}\right) \mathscr B \subset \mathscr B \subseteq \Omega.
\end{equation}
\end{lemma}
\begin{proof}
We first observe that if $\rho$ is small enough depending on $M$, then $B(P, \rho)$ appears convex in the coordinates $(x^\mu)$.
Indeed, $(x^\mu)$ is $O(\rho)$-close in the $C^2$ topology to a normal coordinate system $(y^\mu)$ based at $P$, so it suffices to show that $B(P, \rho)$ appears uniformly strictly convex as $\rho \to 0$ in the coordinates $(y^\mu)$.
However, in those coordinates $(y^\mu)$, $\partial B(P, \rho)$ appears to be a sphere of radius $\sim \rho$, so its principal curvatures appear to be $\gtrsim \rho^{-1}$, implying the uniformly strict convexity.

By \cite[Theorem 4.8]{Giusti77}, the convexity in coordinates of $B(P, \rho)$ and (\ref{rep as a good graph hyp}), there exists $\Omega \subseteq \{x^0 = 0\}$ and $u \in C^1(\Omega)$ such that $\partial U$ is the graph of $u$ and 
$$\|\dif u\|_{C^0} \leq \sup_{x_1, x_2 \in \Omega} \frac{|u(x_1) - u(x_2)|}{|x_1 - x_2|} \leq e^{o(c_1^2)}\sqrt{1 - e^{-o(c_1^2)}} \leq c_1.$$

It remains to construct $\mathscr B$ satisfying (\ref{rep as a good graph set nests}); we adapt the proof of \cite[(6.25)]{Giusti77} to this setting.
We may assume that $\Omega^\alpha$ is nonempty; otherwise just choose any ball $\mathscr B$ in the open set $\Omega$.
Let $y := x^0$ and $x := (x^1, \dots, x^{d - 1})$ denote the coordinate functions.
Since $\Omega^\alpha$ is nonempty, there exists $(x_*, y_*) \in \partial U \cap B(P, \alpha \rho)$, and then for $(x, y) \in \partial U$, we obtain from the bound on $\dif u$ that $|y - y_*| \lesssim c_1 \rho$.
Writing $(x_\natural, y_\natural)$ for the coordinates of $P$, we see from the Pythagorean theorem and the approximate euclideanness of the metric that for $(x, y) \in \partial B(P, \rho)$,
$$|x - x_\natural|^2 + (y - y_\natural)^2 = \rho^2 + O(\rho^3).$$
If we additionally impose $(x, y) \in \partial U$, then it follows that 
$$|x - x_\natural|^2 + (y_* - y_\natural)^2 = \rho^2 + O(c_1 \rho^2)$$
at least if $\rho$ is small depending on $c_1$.
But $(x, y) \in \partial U \cap \partial B(P, \rho)$ iff $x \in \partial \Omega$ and $y = u(x)$.
So $\partial \Omega$ lies in an annulus:
$$\partial \Omega \subset \{x: (1 - Cc_1)\rho^2 - (y_* - y_\natural)^2 \leq |x - x_\natural|^2 \leq (1 + Cc_1)\rho^2 - (y_* - y_\natural)^2\}$$
for some constant $C$ depending on $M$ only.
A similar argument with $(x, y) \in \partial U \cap \partial B(P, \alpha \rho)$ (instead of $\partial B(P, \rho)$) shows that
$$\partial \Omega^\alpha \subset \{x: (1 - Cc_1) \alpha^2 \rho^2 - (y_* - y_\natural)^2 \leq |x - x_\natural|^2 \leq (1 + Cc_1)\alpha^2 \rho^2 - (y_* - y_\natural)^2\}.$$
If we take $c_1$ small enough depending on $C, c_0$, then we may choose
$$\mathscr B := \mathscr B\left(x_\natural, \sqrt{(1 - Cc_1)\rho^2 - (y_* - y_\natural)^2}\right).$$
Then $\mathscr B \subseteq \Omega$ and since 
\begin{align*}
\left(\alpha + \frac{c_0}{2}\right) \sqrt{(1 - Cc_1)\rho^2 - (y_* - y_\natural)^2} &\geq \sqrt{(1 + Cc_1)\alpha^2 \rho^2 - (y_* - y_\natural)^2},
\end{align*}
we also have $(\alpha + c_0/2) \mathscr B \supseteq \Omega^\alpha$.
\end{proof}

\begin{proof}[Proof of Lemma \ref{Miranda44}]
Throughout this proof we assume that there exists some $Q \in \partial U \cap B(P, \alpha \rho)$.
If not, then (\ref{Miranda44 concl}) is vacuous since then $\Exc_{\alpha \rho} (U, P) = 0$.

Consider normal coordinates $(x^\mu)$ based at $Q$, such that at $Q$, $Y := \partial_0$ points in the same direction as $X$.
If $\rho$ is small enough depending on $c_1$, then by (\ref{Miranda44 normal hyp}), we have 
$$(\normal_U, Y) \geq e^{-o(c_1^2)}.$$
By Lemma \ref{hypersurfaces are graphical}, it follows that $\partial U$ is $Y$-graphical over some set $\Omega \subseteq \{x^0 = 0\}$, with graphical function $u$ satisfying $\|\dif u\|_{C^0} \leq c_1$.
Moreover, there exists a ball $\mathscr B \subseteq \Omega$ satisfying (\ref{rep as a good graph set nests}).
Let $J_t := \{(x, y) \in B(P, \rho): x \in t\mathscr B\}$ be the cylinder over $t\mathscr B$.

For $t \leq 1$, we can use (\ref{path ordered exponential taylor series}) and Corollary \ref{doubling dimension} to compute in the coordinates $(x^\mu)$ that
\begin{align*}
    \Exc_{J_t}(U, Q) &= |\partial U \cap J_t| - \int_{J_t} \mathcal Pe^{-\int_R^Q \Gamma} \normal_U(R) \dif \mu(R) \\
    &= |\partial U \cap J_t| - \left[\int_{J_t} (\normal_U)_\alpha \dif \mu\right] \dif x^\alpha + O(\rho^{d + 1}).
\end{align*}
Here $\Gamma$ are Christoffel symbols and $\mu$ is the surface measure.

Let $\nu$ denote the surface measure with respect to the euclidean metric on $M$ obtained from coordinates, and let $\normal_U'$ denote the conormal $1$-form with respect to the euclidean metric.
On \cite[pg83]{Giusti77} it is shown that
\begin{equation}\label{Excess versus Lagrangian}
\left|\left[\int_{J_t} (\normal_U')_\alpha \dif \nu\right] \dif x^\alpha\right| = \Japan{\avg_{t\rho} \dif u} |t\mathscr B|.
\end{equation}
By Corollary \ref{doubling dimension} and the approximate euclideanness of the metric, we can replace $\dif \nu$ with $\dif \mu$, and $\normal_U'$ with $\normal_U$, in (\ref{Excess versus Lagrangian}), at the price of an error of size $O(\rho^{d + 1})$.
So from (\ref{approximate by euclidean lagrangian}), we deduce 
\begin{align*}
\Exc_{J_t}(U, P)
&= \Exc_{J_t}(U, Q) + O(\rho^{d + 1}) \\
&= |\partial U \cap J_t| - \int_{t\mathscr B} \Japan{\avg_{t\rho} \dif u} \dif x + O(\rho^{d + 1}) \\
&= \int_{t\mathscr B} \Lagrange(u, \dif u) - \Lagrange(u, \avg_{t\rho} \dif u) + O(\rho^{d + 1}).
\end{align*}
Applying (\ref{rep as a good graph set nests}) and the properties (\ref{approximate monotone}) and (\ref{translation invariant excess}) of the excess, we obtain from Lemma \ref{Miranda43} with $\tilde \alpha := \alpha + c_0/2$ that
\begin{align*}
\Exc_{\alpha \rho}(U, P) 
&\leq \Exc_{J_{\alpha + c_0/2}}(U, P) + O(\rho^{d + 1}) \\
&\leq \int_{(\alpha + c_0/2) \mathscr B} \Lagrange(u, \dif u) - \Lagrange(u, \avg_{(\alpha + c_0)\rho} \dif u) + O(\rho^{d + 1}) \\
&\leq \left(\left(\alpha + \frac{c_0}{2}\right)^{d + 1} + \frac{c_0}{2}\right) \int_{\mathscr B} \Lagrange(u, \dif u) - \Lagrange(u, \avg_\rho \dif u) + O(\rho^{d + 1}).
\end{align*}
Since $\alpha < 1$, if $c_0$ is smaller than an absolute constant, $((\alpha + c_0/2)^{d + 1} + c_0/2) \leq \alpha^{d + 1} + c_0$.
We can then reverse the above computation to obtain 
\begin{align*}(\alpha^{d + 1} + c_0) \int_{\mathscr B} \Lagrange(u, \dif u) - \Lagrange(u, \avg_\rho \dif u) &\leq (\alpha^{d + 1} + c_0) \Exc_\rho(U, P) + O(\rho^{d + 1}). \qedhere 
\end{align*}
\end{proof}

\section{Mollification} \label{Mollifiers}
In this section we prove Lemma \ref{single mollify} by following \cite[Chapter 7]{Giusti77}.
The main point is to estimate a certain convolution operator which mollifies the set of least perimeter $U$ in such a way that, if $\normal_U$ is close to a given vector field $X$ on average, then its mollification is close to $X$ in $C^0$.
More precisely, given normal coordinates $(x^\mu)$ based at some point $P$, and a distribution $f$, we define 
\begin{equation}\label{convolution operator}
f_\varepsilon(x) := \frac{d + 1}{|\Ball^d| \varepsilon^d} \int_{\mathcal B(x, \varepsilon)} f(y) \left(1 - \frac{|x - y|}{\varepsilon}\right) \dif y.
\end{equation}
Here and always in this section, $\mathcal B(x, \varepsilon)$ is the \emph{euclidean} ball taken in the coordinates $(x^\mu)$.

It is possible to define an intrinsic operator, by replacing $|x - y|$ with $\dist(x, y)$, $\mathcal B$ with $B$, and $\dif y$ with $\dif V(y)$.
However, such an operator turns out to be rather annoying to work with, because it does not commute with $\dif$, so we shall not take this approach here.\footnote{Miranda \cite{Miranda66} uses a different convolution operator, which has an invariant interpretation as the pseudodifferential operator $(1 - \varepsilon^{-2} \Delta)^{-1}$, but sadly seems tricky to estimate in the presence of curvature.}

\subsection{Pointwise estimate on the normal vector}
When applied to sets of least perimeter with small excess, the convolution operator (\ref{convolution operator}) satisfies an analogue of \cite[Theorem 7.3]{Giusti77}, which in the flat case asserted that the normal vector of a convolved set of least perimeter is pointwise close to an aligned vector field; we now prove the same result.

\begin{lemma}\label{main mollifier lemma}
Let $\rho, \gamma > 0$ be small enough depending on $M, P$, and let $U$ be a set of least perimeter such that
\begin{equation}\label{hypothesis on main mollifier lemma}
\Exc_\rho(U, P) \leq \gamma \rho^{d - 1}.
\end{equation}
Furthermore, suppose that $\rho^2 \leq \gamma$.
Let $\varepsilon := \gamma^4\rho$, $\sigma := \gamma^{1/(2d)} \rho$, and $\varphi := (1_U)_\varepsilon$, taken in normal coordinates $(x^\mu)$ based at $P$.
Then there exists an aligned vector field $X$ such that for every $y \in (O(\gamma^2), 1 - O(\gamma^2))$, the level set $N_y := \partial \{\varphi > y\} \cap B(P, \rho - 2\sigma)$ is a $C^1$ hypersurface, and on $N_y$,
\begin{equation}\label{claim on main mollifier lemma}
    (\normal_{N_y}, X) \geq e^{-O(\rho^2 + \sigma)}.
\end{equation}
\end{lemma}

We begin the proof of Lemma \ref{main mollifier lemma} by letting $\delta := \gamma^d$ and $u := 1_U$.
We greedily construct a cover $\mathcal V = \{V_n: 1 \leq n \leq N\}$ of $\partial^* U \cap B(P, \varepsilon(1 - \delta))$, where $V_n := B(Q_n, 2\delta\varepsilon)$.
Since it was constructed greedily, $\mathcal V$ is \dfn{efficient} in the sense that there exists $C > 0$ only depending on $M$, such that any $Q \in B(P, \varepsilon(1 - \delta))$ is contained in $\leq C$ elements of $\mathcal V$.
Let $V_0$ be the annulus $B(P, \varepsilon) \setminus B(P, \varepsilon(1 - 2\delta))$.

To construct the aligned vector field $X$, let $v$ be the unit vector pointing in the same direction as $(\avg_{B(P, \rho), P, \mu} \normal_U)^\sharp$ where $\mu$ is the surface measure on $\partial^* U$; let $X$ be the aligned vector field extending $v$.
It then follows from (\ref{path ordered exponential taylor series}), Corollary \ref{doubling dimension}, and the fact that $\rho^2 \leq \gamma$ that 
\begin{equation}\label{hypothesis on mollifier sublemma}
\int_{B(P, \rho)} (|\dif u| - Xu)(z) \dif z = \Exc_\rho(U, P) + O(\rho^{d + 1}) \lesssim \gamma \rho^{d - 1}.
\end{equation}
We use this fact to estimate the convolution of $u$ in each ball $Q_n$.
However, $X$ is technically inconvenient as $|\dif u| - Xu$ can be negative. 
So we introduce $\tilde X := X/|X|$.

\begin{lemma}\label{mollifier sublemma}
Let $\tilde X := X/|X|$.
For every $Q \in \partial^* U \cap B(P, \varepsilon)$ and $x \in B(P, \rho - 2\sigma)$,
$$(1_{B(Q, 2\delta\varepsilon)}(|\dif u| - \tilde Xu))_\varepsilon(x) \lesssim \gamma^{1/(2d)} (1_{B(Q, \delta\varepsilon)} |\dif u|)_\varepsilon(x).$$
\end{lemma}
\begin{proof}
Let $\chi_\varepsilon$ be the convolution kernel.
Throughout this proof we may assume that $\gamma < 1/10$, so that $\sigma > \varepsilon$ and $\delta < 1/100$.
Then every point considered in this proof will be contained in the ball $B(P, 2\sigma)$, in which the metric takes the form $g_{\mu \nu} = \delta_{\mu \nu} + O(\sigma^2)$.
In particular, the volume form is $e^{O(\sigma^2)} \dif x$.

Let $V := B(Q, 2\delta\varepsilon)$; then we want to estimate 
$$(1_V(|\dif u| - \tilde Xu))_\varepsilon(x) = \int_{\mathcal B(x, \varepsilon) \cap V} \chi_\varepsilon(x - y) (|\dif u| - \tilde Xu)(y) \dif y.$$
Since $|\dif u| - \tilde Xu \geq 0$,
\begin{equation}\label{begin breaking up the mollifier}
\int_{\mathcal B(x, \varepsilon) \cap V} \chi_\varepsilon(x - y) (|\dif u| - \tilde Xu)(y) \dif y \leq \sup_{y \in V} \chi_\varepsilon(x - y) \int_V (|\dif u| - \tilde Xu)(z) \dif z.
\end{equation}
Moreover, by \cite[91]{Giusti77},
$$\sup_{y \in V} \chi_\varepsilon(x - y) \lesssim \inf_{y \in B(Q, \delta\varepsilon)} \chi_\varepsilon(x - y).$$

Let $W := B(Q, \sigma)$; by (\ref{weak monotonicity}), we obtain 
$$(2\delta\varepsilon)^{1 - d} \int_V |\dif u|(z) \dif z \leq e^{O(\sigma^2)} \sigma^{1 - d} \int_W |\dif u|(z) \dif z.$$
We use this to bound
\begin{align*}
(2\delta\varepsilon)^{1 - d} \int_V (|\dif u| - \tilde X u)(z) \dif z
&\leq \sigma^{1 - d} \int_W (|\dif u| - X u)(z) \dif z + O(\sigma^{3 - d}) \int_W |\dif u|(z) \dif z \\
&\qquad + \sigma^{1 - d} \int_W X u(z) \dif z - (2\delta\varepsilon)^{1 - d} \int_V Xu(z) \dif z \\
&\qquad + (2\delta\varepsilon)^{1 - d} \int_V (\tilde X - X)u(z) \dif z\\
&=: I_1 + I_2 + I_3 - I_4 + I_5.
\end{align*}
We then use (\ref{hypothesis on mollifier sublemma}) to bound
$$I_1 \leq \sigma^{1 - d} \int_{B(P, \rho)} (|\dif u| - X u)(z) \dif z \lesssim \gamma^{\frac{1 - d}{2d} + 1} \leq \gamma^{1/(2d)}.$$
By Corollary \ref{doubling dimension} we have $I_2 \lesssim \gamma^{1/(d - 1)}$, and since $X - \tilde X = O(\varepsilon^2)$ on $V \subseteq B(P, O(\varepsilon))$, $I_5 \lesssim \varepsilon^2 \lesssim \gamma^8$.

To estimate $I_3 - I_4$, we recall the notation (\ref{integral of du}) for vector-valued integrals.
For any $r > 0$, we can use (\ref{path ordered exponential taylor series}), the Taylor expansion of the metric, and Corollary \ref{doubling dimension} to estimate
$$r^{1 - d} \int_{B(Q, r)} Xu(z) \dif z = r^{1 - d} \left[\int_{B(Q, r)} K(Q, z) \dif u(z) \dif z\right]_0 + O(r^2) = I(u, Q, r)_0 + O(r^2).$$
Therefore we can apply the monotonicity formula, Proposition \ref{Monotone}, to compute
\begin{align*}
    I_3 - I_4 & \leq |\sigma^{1 - d} I(u, Q, \sigma) - (2 \delta \varepsilon)^{1 - d} I(u, Q, 2 \delta \varepsilon)| + O(\sigma^2) \\
    &\lesssim \sqrt{1 + (d - 1) \log \frac{\sigma}{2\delta\varepsilon}} \sqrt{\sigma^{1 - d} \int_W \star |\dif u|} \sqrt{\int_{2\delta\varepsilon}^\sigma \partial_r \left[e^{O(r^2)} r^{1 - d} \int_{B(Q, r)} \star |\dif u|\right] \dif r}\\
&\qquad + \sigma^{3 - d} \int_W \star |\dif u| + \sigma^2 \\
&=: J_1 J_2 J_3 + J_4 + J_5.
\end{align*}
Then (using Corollary \ref{doubling dimension} as necessary), we bound $J_1 \lesssim -\log \gamma$, $J_2 \lesssim 1$, and $J_4 \lesssim J_5 = \gamma^{1/(2d)}$.

In order to estimate $J_3$, we introduce a normal coordinate system $(\tilde x^\mu)$ based at $Q$, with $Y = \tilde \partial_0$ chosen so $Y(Q)$ and $X(Q)$ have the same span.
Then $Y(Q) = X(Q)/|X(Q)| = X(Q) + O(\varepsilon^2)$.
Taylor expanding $Y - X$ around $Q$, we obtain $\|Y - X\|_{C^0(W)} \lesssim \sigma$.
Then
\begin{align*}
J_3^2 &\leq \sigma^{1 - d} \int_W |\dif u|(z) \dif z - (2 \delta \varepsilon)^{1 - d} \int_V |\dif u|(z) \dif z + O(\sigma^{3 - d}) \int_W \star |\dif u| \\
&= \sigma^{1 - d} \int_W (|\dif u| - Xu)(z) \dif z + \sigma^{1 - d} \int_W (X - Y)u(z) \dif z + \sigma^{1 - d} \int_W Yu(z) \dif z \\
  &\qquad - (2 \delta\varepsilon)^{1 - d} \int_V Y u(z) \dif z + O(\sigma^{3 - d}) \int_W \star |\dif u| \\
&=: K_1 + K_2 + K_3 - K_4 + K_5.
\end{align*}
Then $K_1 = I_1 \leq \gamma^{1/2}$, $K_2 \lesssim \sigma = \gamma^{1/(2d)}$, and $K_5 = J_4 \lesssim \gamma^{1/(2d)}$.

To estimate $K_3 - K_4$, introduce the closed $d-1$-form $\psi := \dif \tilde x^1 \wedge \cdots \wedge \dif \tilde x^{d - 1}$.
Then
\begin{equation}\label{K3 calculus}
K_3 = \sigma^{1 - d} \int_W \dif u \wedge \psi = \sigma^{1 - d} \int_{U \cap \partial W} \psi.
\end{equation}
Since $W$ is a ball centered on the $(\tilde x^\mu)$-coordinate origin $P$, we can decompose
$$\partial W = \Gamma_+ \cup \Gamma_0 \cup \Gamma_-$$
where $\pm \tilde x^0 > 0$ on the hemispheres $\Gamma_\pm$ and $\Gamma_0$ is the equator.
Then all positive contributions to the integral in the right-hand side of (\ref{K3 calculus}) come from $\Gamma_+$.
However, if we set $W_0 := W \cap \{\tilde x^0 = 0\}$, then $\Gamma_+$ and $W_0$ are homologous relative to their common boundary $\Gamma_0$.
By Stokes' theorem,
$$K_3 \leq \sigma^{1 - d} \int_{\Gamma_+ \cap U} \psi \leq \sigma^{1 - d} \int_{\Gamma_+} \psi = \sigma^{1 - d} \int_{W_0} \psi.$$
Since $\psi$ is the euclidean volume form on $W_0$, and $W_0$ is a $d-1$-ball whose euclidean radius is $\leq \sigma + O(\sigma^3)$, it follows that
\begin{equation}\label{K3 calculus 2}
K_3 \leq |\Ball^{d - 1}| + O(\sigma^2).
\end{equation}
By Corollary \ref{doubling dimension}, the right-hand side of (\ref{K3 calculus 2}) is $\leq K_4 + O(\gamma^{1/d})$.

Adding up all the $K_i$, we finally conclude that $J_3 \lesssim \gamma^{1/d}$ and hence $I_1 + I_2 + I_3 - I_4 \lesssim \gamma^{1/(2d)}$.
Plugging this back into (\ref{begin breaking up the mollifier}),
$$(1_V(|\dif u| - X u))_\varepsilon(x) \lesssim (\delta\varepsilon)^{d - 1} \gamma^{1/(2d)} \inf_{y \in B(Q, \delta\varepsilon)} \chi_\varepsilon(x - y).$$
We finally apply Corollary \ref{doubling dimension} to prove
\begin{align*}
(\delta\varepsilon)^{d - 1} \inf_{y \in B(Q, \delta\varepsilon)} \chi_\varepsilon(x - y)
&\lesssim \int_{B(Q, \delta \varepsilon)} \chi_\varepsilon(x - y) |\dif u|(y) \dif y = |\dif u|_\varepsilon(x). \qedhere
\end{align*}
\end{proof}

\begin{proof}[Proof of Lemma \ref{main mollifier lemma}]
We begin by replacing $X$ with $\tilde X$.
Since $X = \tilde X + O(\varepsilon^2)$ on $B(P, 2\varepsilon)$, we obtain
$$((X - \tilde X)u)_\varepsilon(x) = \int_{\mathcal B(x, \varepsilon)} (X - \tilde X)u(y) \chi_\varepsilon(x - y) \dif y \lesssim \gamma^8 \int_{\mathcal B(x, \varepsilon)} \chi_\varepsilon(x - y) |\dif u|(y) \dif y.$$
In particular,  
$$(|\dif u| - Xu)_\varepsilon(x) = (|\dif u| - \tilde Xu)_\varepsilon(x) + \gamma^8 |\dif u|_\varepsilon(x).$$
By construction, $\supp \dif u \subseteq \bigcup_n V_n$, so
\begin{equation}\label{sum over cover}
(|\dif u| - \tilde Xu)_\varepsilon \leq \sum_{n = 0}^N (1_{V_n} (|\dif u| - \tilde Xu))_\varepsilon.
\end{equation}

It remains to estimate $|\dif u| - \tilde Xu$ on $V_0$.
As on \cite[92]{Giusti77}, $\chi_\varepsilon(x - y) \lesssim \varepsilon^{-d} \delta$ for $y \in V_0$; moreover, $V_0 \subseteq B(P, \varepsilon)$, so by Corollary \ref{doubling dimension},
$$\int_{V_0} (|\dif u| - \tilde Xu)(y) \chi_\varepsilon(x - y) \dif y \lesssim \varepsilon^{-d} \delta \int_{B(x, \varepsilon)} |\dif u|(y) \dif y \lesssim \frac{\delta}{\varepsilon}.$$
If $u(x) \in (O(\gamma^2), 1 - O(\gamma^2))$, then by \cite[Lemma 7.1]{Giusti77}, there exists $z \in \partial^* U \cap B(x, (1 - \gamma)\varepsilon)$.
In particular, $B(z, \gamma\varepsilon/2) \subseteq B(x, (1 - \gamma/2)\varepsilon)$ so on $B(z, \gamma\varepsilon/2)$, $\chi_\varepsilon(x - y) \gtrsim \varepsilon^{-d} \gamma$.
So 
$$\frac{\delta}{\varepsilon} \lesssim \gamma \inf_{y \in B(x, \varepsilon)} \chi_\varepsilon(x - y) \int_{B(x, \varepsilon)} |\dif u(y)| \dif y \leq \gamma |\dif u|_\varepsilon(x).$$
Plugging this estimate and Lemma \ref{mollifier sublemma} into (\ref{sum over cover}), we can apply the efficiency of $\mathcal V$ to see that on
$$\Omega := B(P, \rho - 2\sigma) \cap \{O(\gamma^2) < \varphi < 1 - O(\gamma^2)\}$$
we have
$$(|\dif u| - Xu)_\varepsilon \lesssim \gamma^{1/(2d)} |\dif u|_\varepsilon$$
which implies (\ref{claim on main mollifier lemma}).
In particular, near $\Omega$, one has $|\dif \varphi| > 0$.
Therefore $\varphi$ is a $C^1$ submersion by \cite[Lemma 7.1]{Giusti77}, which completes the proof.
\end{proof}

\subsection{Application to the de Giorgi lemma}
We now can complete the proof of the de Giorgi lemma.
Aside from the fact that we use compactness-and-contradiction to obtain the result for a set of least perimeter, rather than a sequence of sets of least perimeter as in \cite{Giusti77, Miranda66}, the proof is almost identical to \cite[Lemma 7.5]{Giusti77}, so we just sketch it.

\begin{proof}[Proof of Lemma \ref{single mollify}]
Suppose not.
Then there exist balls $B_n := B(P_n, \rho_n)$ and sets $U_n$ of least perimeter in $B_n$ such that
\begin{equation}\label{single mollify Excess assumption}
\gamma_n := \rho_n^{1 - d} \Exc_{\rho_n} (U_n, P_n) \leq n^{-2},
\end{equation}
and $\rho_n \leq 1/n$, but such that for every open set $V_n \subseteq B(P_n, (1 - \varepsilon) \rho_n)$ with $C^1$ boundary, and every $P_n$-aligned vector field $X_n$, at least one of (\ref{single mollify normal}), (\ref{single mollify minimality}), or (\ref{single mollify excess}) fail.

Now let $w_n := (u_n)_{\gamma_n^4 \rho_n}$, where the convolution was taken with respect to normal coordinates based at $P_n$.
Draw $t \in [1 - c_1, 1]$ uniformly at random.
Applying Lemma \ref{main mollifier lemma} with $\gamma := n^{-2}$ and $\rho := \rho_n$ (so $\gamma \geq \rho^2$), we find $a_n = O(\gamma_n^2)$, $b_n = 1 - O(\gamma_n^2)$, and a $P_n$-aligned vector field $X_n$ such that for $y \in (a_n, b_n)$, the level sets $\{w_n = y\}$ have normal vector $C^0$-close to $X_n$.
In particular, by the coarea formula,
$$\int_{tB_n} \star |\dif w_n| = \int_0^1 |\partial^* \{w_n > y\} \cap tB_n| \dif y \geq \int_{a_n}^{b_n} |\partial^* \{w_n > y\} \cap tB_n| \dif y,$$
so by the mean value theorem, there exists $y_n \in (a_n, b_n)$ such that
$$|\partial^* \{w_n > y_n\} \cap tB_n| \leq \frac{1}{b_n - a_n} \int_{tB_n} \star |\dif w_n|.$$
We then set $V_n := \{w_n > y_n\}$, $v_n := 1_{V_n}$, so that $V_n$ has $C^1$ boundary in $tB_n$ and
\begin{equation}\label{MVT mollifier}
|V_n \cap tB_n| \leq \frac{1}{b_n - a_n} \int_{tB_n} \star |\dif w_n|.
\end{equation}
By (\ref{claim on main mollifier lemma}), $(\normal_{V_n}, X_n) \geq 1 - O(n^{-1/d})$, so for $n$ large enough depending on $c_1$, $V_n$ satisfies (\ref{single mollify normal}).

Let $\Gamma_n := \partial(tB_n)$ and draw $s \in [0, t]$ at random.
Arguing completely analogously to the proofs of \cite[(7.23), (7.22)]{Giusti77}, respectively, we see that almost surely,
\begin{align}
\|u_n - v_n\|_{L^1(\Gamma_n)} &\ll \gamma_n \label{trace of vn} \\
|\partial V_n \cap sB_n| &\leq |\partial^* U_n \cap sB_n| + \gamma_n. \label{difference of surface area}
\end{align}
The conjunction of (\ref{trace of vn}), (\ref{difference of surface area}), and (\ref{a priori estimate 1}) implies
\begin{equation}
||\partial^* U_n \cap tB_n| - |\partial V_n \cap tB_n|| \ll \gamma_n, \label{mollifier quant2}
\end{equation}
and the conjunction of (\ref{mollifier quant2}), (\ref{a priori estimate 1}), the fact that $U_n$ has least perimeter, (\ref{single mollify Excess assumption}) and (\ref{trace of vn}) implies (\ref{single mollify minimality}) if $n$ is large enough depending on $c_1$.

To derive a contradiction, we must show that $V_n$ satisfies (\ref{single mollify excess}) for $n$ large enough depending on $c_1$.
Let $\varpi \leq t\rho$, and estimate
\begin{align*}
    |\Exc_\varpi(U_n, P_n) - \Exc_\varpi(V_n, P_n)|
    &\leq ||\partial^* U_n \cap t B_n| - |\partial V_n \cap t B_n||\\
    &\qquad + \left|\left[\int_{B(P_n, \varpi)} \star \partial_\mu(u_n - v_n) \right] \dif x^\mu(P_n)\right| + O(\rho_n^{d + 1}) \\
    &=: I_1 + I_2 + I_3.
\end{align*}
By (\ref{mollifier quant2}), $I_1 \leq c_1 \Exc_\varpi(U_n, P_n)/3$ if $n$ is large, and $I_3$ is irrelevant.
By Stokes' theorem and (\ref{trace of vn}), if $n$ is large then
\begin{align*}
    I_2 &\leq \left|\left[\int_{\partial B(P, \varpi)} (\normal_{B_\varpi})_\mu (u_n - v_n) \dif S_{\partial B(P, \varpi)}\right] \dif x^\mu(P_n)\right| \leq \frac{c_1}{3} \Exc_\varpi(U_n, P_n) + O(\rho_n^{d + 1}).
\end{align*}
This implies (\ref{single mollify excess}) for $n$ large and so contradicts our assumptions.
\end{proof}

\printbibliography

\end{document}